\documentclass[11pt,reqno]{amsart}
\usepackage{amssymb,amsthm,amsmath,tikz}

\begin{document}

\newtheorem*{unnumthm}{Theorem}
\newtheorem{thm}{Theorem}[section]
\newtheorem{cor}[thm]{Corollary}
\newtheorem{lem}[thm]{Lemma}
\newtheorem{prop}[thm]{Proposition}
\theoremstyle{definition}
\newtheorem{defn}[thm]{Definition}
\newtheorem{ex}[thm]{Example}

\newcommand{\bm}[1]{\mbox{\boldmath $#1$}}
\newcommand{\sr}{\mbox{$\mathcal{R}$}}
\renewcommand{\sc}{\mbox{$\mathcal{C}$}}
\newcommand{\R}{\mbox{${R}$}}
\newcommand{\C}{\mbox{${C}$}}
\renewcommand{\L}{\mbox{$\lambda$}}
\newcommand{\sq}{\mbox{$\lozenge$}}
\newcommand{\h}{\mbox{$h$}}

\renewcommand{\subjclassname}{\textup{2010} Mathematics Subject Classification}

\title{Sort-Invariant Non-Messing-Up}

\author{Bridget Eileen Tenner}
\email{bridget@math.depaul.edu}
\address{Department of Mathematical Sciences, DePaul University, Chicago, Illinois, USA}


\begin{abstract}
A poset has the non-messing-up property if it has two covering sets of disjoint saturated chains so that for any labeling of the poset, sorting the labels along one set of chains and then sorting the labels along the other set yields a linear extension of the poset.  The linear extension yielded by thus twice sorting a labeled non-messing-up poset may be independent of which sort was performed first.  Here we characterize such sort-invariant labelings for convex subposets of a cylinder.  They are completely determined by avoidance of a particular subpattern: a diamond of four elements whose smallest two labels appear at opposite points.
\end{abstract}

\maketitle

\section{Introduction}

The non-messing-up property of matrices is a well-known result about sorting the entries in a matrix.  It appears in work by Boerner \cite{boerner}, Gale and Karp \cite{gale-karp}, Knuth \cite{knuth}, and Winkler \cite{winkler}, and a detailed discussion of its provenance occurs in \cite{nmut}.  Given a matrix $M$, let $\sc(M)$ be the matrix obtained by ordering the entries within each column of $M$, and likewise let $\sr(M)$ be the matrix obtained by ordering the entries within each row of $M$.

\begin{ex}\label{ex:sorting}
Let
$$M = \begin{bmatrix}4 & 9 & 7 & 8\\12 & 5 & 1 & 10\\2 & 6 & 11 & 3\end{bmatrix}.$$
Then the matrices $\sc(M)$ and $\sr(M)$ are
$$\sc(M) = \begin{bmatrix}2 & 5 & 1 & 3\\4 & 6 & 7 & 8\\12 & 9 & 11 &10\end{bmatrix} \hspace{.25in} \text{and} \hspace{.25in} \sr(M) = \begin{bmatrix}4 & 7 & 8 & 9\\1 & 5 & 10 & 12\\2 & 3 & 6 &11\end{bmatrix}.$$
Moreover,
\begin{equation}\label{eqn:array-srsc}
\sr\sc(M) = \begin{bmatrix}1 & 2 & 3 & 5\\4 & 6 & 7 & 8\\9 & 10 & 11 & 12\end{bmatrix} \hspace{.25in} \text{and} \hspace{.25in} \sc\sr(M) = \begin{bmatrix}1 & 3 & 6 & 9\\2 & 5 & 8 & 11\\4 & 7 & 10 & 12\end{bmatrix}.
\end{equation}

\end{ex}

\begin{thm}[Non-messing-up matrices]\label{thm:nmut array}
Consider a matrix $M$.  The entries in each column in the matrix $\sr\sc(M)$ are in non-decreasing order; that is, $\sr\sc(M) = \sc\sr\sc(M)$.
\end{thm}

Less formally, Theorem~\ref{thm:nmut array} says that after sorting the data within each column of a matrix, sorting the data within each row of the subsequent matrix does not ``mess up'' the fact that each column's data are sorted.  Note that by symmetry, the entries in the rows of the matrix $\sc\sr(M)$ must also be sorted, so $\sc\sr(M) = \sr\sc\sr(M)$ as well.

The expressions of \eqref{eqn:array-srsc} provide an example of the non-messing-up property of matrices.  Note from these that it is not necessarily the case that the matrices $\sr\sc(M)$ and $\sc\sr(M)$ coincide; that is, the sorting operations do not necessarily commute.  It is this issue which we address here.  Moreover, we examine it for more general non-messing-up posets, which we previously treated in \cite{nmut}.

In Section~\ref{sec:nmu} we review the results of \cite{nmut} and define the central object of this paper: a transverse non-messing-up poset.  Our main result, stated precisely in Theorem~\ref{thm:sort-invariance}, is that the sorting operations commute on a labeling of a transverse non-messing-up poset if and only if the labeling avoids a particular forbidden subpattern (subject to some degeneracies) consisting of a diamond of four elements whose smallest labels appear at opposite points.   Unfortunately this result does not hold for non-transverse non-messing-up posets, as shown in Section~\ref{sec:non-transverse}.  The paper concludes with a brief discussion of sorted matrices and when they have a preferred sorting procedure; that is, when it is more likely that a sorted matrix arises from first column-sorting and then row-sorting, or vice versa.

\section{Non-messing-up posets}\label{sec:nmu}

The goal of this paper is to describe exactly when the sorting operations $\sr$ and $\sc$ commute on a labeling, and we do this in the context of non-messing-up posets.  Previously, in \cite{nmut}, we defined and characterized a generalization of the non-messing-up phenomenon to the setting of posets.  That work is summarized here, and more details can be found in \cite{nmut}.

Generalizing to the realm of posets is natural because an $r$-by-$c$ matrix resembles the poset $\bm{r} \times \bm{c}$, a product of two chains.  The rows and columns of the matrix correspond to two sets of disjoint saturated chains covering the elements of $\bm{r} \times \bm{c}$.  Sorting a chain means putting that chain's labels in order so that the minimum element in the chain has the minimum label.  Thus the non-messing-up property indicates that sorting the labels along both the rows and the columns gives a linear extension of $\bm{r} \times \bm{c}$.

\begin{defn}
Let $P$ be a finite poset.  A \emph{gridwork} for $P$ is a pair $(\R,\C)$ where $\R$ and $\C$ are each sets of disjoint saturated chains covering the elements of $P$, and where each covering relation in $P$ is contained in an element of $\R \cup \C$.  We call the elements of $\R$ the \emph{rows} of $P$, and the elements of $\C$ the \emph{columns} of $P$.
\end{defn}

\begin{defn}
Let $P$ be a finite poset with gridwork $(\R,\C)$.  The gridwork is \emph{transverse} if each row and each column intersect at most once.
\end{defn}

\begin{defn}
Let $P$ be a finite poset with a transverse gridwork $(\R,\C)$.  Given a labeling of the elements of $P$, let $\sr$ be the operation which sorts the labels within each row of the poset, so that the minimum element in a row has the minimum label of that row.  Likewise, let $\sc$ be the operation which sorts the labels within each column of the poset.
\end{defn}

We are now able to carry the notion of non-messing-up to the setting of posets.

\begin{defn}\label{defn:nmu}
A poset $P$ has the \emph{transverse non-messing-up} property, or is \emph{transverse non-messing-up}, if there exists a transverse gridwork such that, for any labeling $\L$ of the elements of $P$, both $\sr\sc(\L)$ and $\sc\sr(\L)$ yield linear extensions of $P$; that is, $\sr\sc(\L) = \sc\sr\sc(\L)$ and $\sc\sr(\L) = \sr\sc\sr(\L)$.
\end{defn}

Before stating the characterization of transverse non-messing-up posets found in \cite{nmut}, we must state one more definition.

\begin{defn}
Fix positive integers $k$ and $n$ so that $k < n$.  The \emph{cylinder poset} $Cyl_{k,n}$ is $\mathbb{Z}^2/(-k,n-k)\mathbb{Z}$. The partial order on $Cyl_{k,n}$ is induced by the componentwise partial order on $\mathbb{Z}^2$.
\end{defn}

\begin{thm}[\cite{nmut}]\label{thm:nmut}
A poset is transverse non-messing-up if and only if each of its connected components $P$ is a convex subposet of a cylinder poset.
\end{thm}

The gridwork for a non-messing-up poset is specified in \cite{nmut}, and the sets $\R$ and $\C$ are analogous to the rows and columns of the motivating matrix example.  The reader is encouraged to find the details in \cite{nmut}.

Note that one can also consider non-transverse non-messing-up posets, and those are also classified in \cite{nmut}.  Such posets differ from transverse non-messing-up posets by allowing a row and a column to intersect more than once.  Each non-transverse non-messing-up poset is obtained from a transverse one by replacing one or more elements by chains, subject to length restrictions.  Not only does a non-transverse non-messing-up poset look quite different from the motivating matrix situation, but there is some redundancy in the sorting operations since, for example, chains of elements in the columns will already have been sorted by $\sr$.

\section{Transverse sort-invariance}

Throughout this section, let $P$ be a transverse non-messing-up poset with transverse gridwork $(\R,\C)$.  For the sake of clarity, suppose that in any labeling of $P$, the elements are labeled by $\{1, 2, \ldots, |P|\}$.

\begin{defn}
A labeling $\L$ of the poset $P$ is \emph{sort-invariant} if $\sr\sc(\L) = \sc\sr(\L)$.
\end{defn}

In other words, a labeling $\L$ is sort-invariant if the operations $\sr$ and $\sc$ commute on $\L$.  Our goal is to characterize sort-invariant labelings of transverse non-messing-up posets.

\begin{defn}
Let $x$ and $y$ be elements of non-messing-up poset, that are in neither the same row nor the same column.  Suppose that $x$ is in row $\bm{r_x}$ and column $\bm{c_x}$, while $y$ is in row $\bm{r_y}$ and column $\bm{c_y}$ (thus $\bm{r_x} \neq \bm{r_y}$ and $\bm{c_x} \neq \bm{c_y}$).  The \emph{corner-set of $x$ and $y$} is the subset
\begin{eqnarray}
\label{eqn:corner-set} \sq(x,y) &=& (\bm{r_x} \cup \bm{r_y}) \cap (\bm{c_x} \cup \bm{c_y}) \\
\nonumber &=& (\bm{r_x} \cap \bm{c_x}) \cup (\bm{r_y} \cap \bm{c_y}) \cup (\bm{r_x} \cap \bm{c_y}) \cup (\bm{c_x} \cap \bm{r_y}).
\end{eqnarray}
If $x$ and $y$ are in either the same row or the same column, then set $\sq(x,y) = \emptyset$.
\end{defn}

Because $P$ is transverse, if $x$ and $y$ are in neither the same row nor the same column in $P$, then we have
\begin{equation}\label{eqn:corner-set-small} 
\sq(x,y) = \{x,y\} \cup (\bm{r_x} \cap \bm{c_y}) \cup (\bm{c_x} \cap \bm{r_y}),
\end{equation}
and this $\sq(x,y)$ contains at most four elements.

The nomenclature here refers to the fact that the set $\sq(x,y)$ looks like the four corners of a diamond containing $x$ and $y$ in $P$, when the edges of the diamond must be either rows or columns in the poset.  Note that because non-messing-up posets (and their rows and columns) arise from convex subposets of cylinder posets, if $x$ and $y$ are comparable in $P$, then both $\bm{r_x} \cap \bm{c_y}$ and $\bm{c_x} \cap \bm{r_y}$ are nonempty.

In the case of the product of two chains, we will show that a labeling is sort-invariant if and only if corner-sets with extreme values on opposite sides are avoided.  In general, a corner-set may have fewer than four elements, and we must avoid more particular patterns with these smaller sets, as stated in the second and third points of Definition~\ref{defn:good}.

\begin{defn}\label{defn:good}
A nonempty corner-set $\sq(x,y)$ is \emph{bad for $\L$} (or simply \emph{bad} if $\L$ is clear from the context) if $|\sq(x,y)| > 2$ and one of the following situations is satisfied:
\begin{itemize}
\item $\sq(x,y) = \{x,y,w,z\}$ and
$$\L(x),\L(y) > \L(w),\L(z) \text{\ \ \ or \ \ } \L(x),\L(y) < \L(w),\L(z), $$
\item $\sq(x,y) = \{x,y,z\}$ with $x,y<z$ and $\L(x), \L(y) >\L(z)$, or
\item $\sq(x,y) = \{x,y,z\}$ with $x,y>z$ and $\L(x), \L(y) <\L(z)$.
\end{itemize}
If $\sq(x,y)$ is not bad for $\L$, then it is \emph{good for $\L$} (or simply \emph{good} if $\L$ is clear from the context).
\end{defn}

\begin{ex}
The following four posets and labelings have only bad corner-sets.  
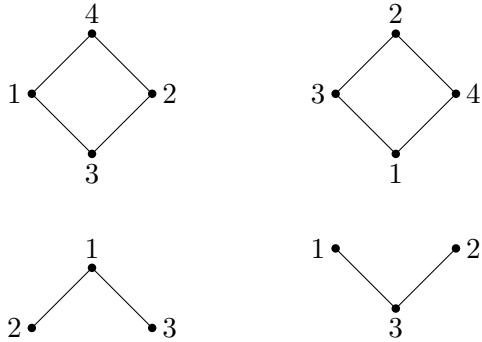
\begin{center}
\begin{tikzpicture}[scale=.8]
\draw (0,0) -- (1, 1);
\draw (1,1) -- (2,0);
\draw (2,0) -- (1, -1);
\draw (1, -1) -- (0,0);
\fill (0,0) circle (2pt) node[left] {$1$};
\fill (1,1) circle (2pt) node[above] {$4$};
\fill (2,0) circle (2pt) node[right] {$2$};
\fill (1,-1) circle (2pt) node[below] {$3$};
\end{tikzpicture}
\hspace{.5in}
\begin{tikzpicture}[scale=.8]
\draw (0,0) -- (1, 1);
\draw (1,1) -- (2,0);
\draw (2,0) -- (1, -1);
\draw (1, -1) -- (0,0);
\fill (0,0) circle (2pt) node[left] {$3$};
\fill (1,1) circle (2pt) node[above] {$2$};
\fill (2,0) circle (2pt) node[right] {$4$};
\fill (1,-1) circle (2pt) node[below] {$1$};
\end{tikzpicture}\\
\ \\
\begin{tikzpicture}[scale=.8]
\draw (0,0) -- (1, 1);
\draw (1,1) -- (2,0);
\fill (0,0) circle (2pt) node[left] {$2$};
\fill (1,1) circle (2pt) node[above] {$1$};
\fill (2,0) circle (2pt) node[right] {$3$};
\end{tikzpicture}
\hspace{.5in}
\begin{tikzpicture}[scale=.8]
\draw (0,1) -- (1, 0);
\draw (1,0) -- (2,1);
\fill (0,1) circle (2pt) node[left] {$1$};
\fill (1,0) circle (2pt) node[below] {$3$};
\fill (2,1) circle (2pt) node[right] {$2$};
\end{tikzpicture}
\end{center}
\end{ex}

\begin{ex}
The following six posets and labelings have only good corner-sets.
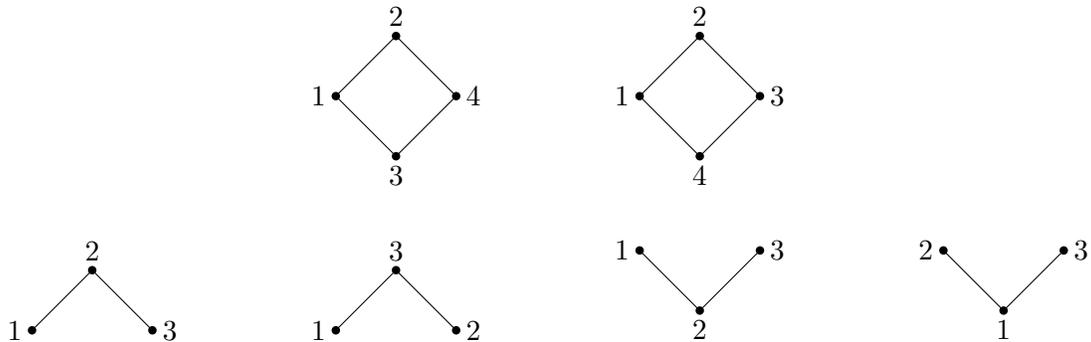
\begin{center}
\begin{tikzpicture}[scale=.8]
\draw (0,0) -- (1, 1);
\draw (1,1) -- (2,0);
\draw (2,0) -- (1, -1);
\draw (1, -1) -- (0,0);
\fill (0,0) circle (2pt) node[left] {$1$};
\fill (1,1) circle (2pt) node[above] {$2$};
\fill (2,0) circle (2pt) node[right] {$4$};
\fill (1,-1) circle (2pt) node[below] {$3$};
\end{tikzpicture}
\hspace{.5in}
\begin{tikzpicture}[scale=.8]
\draw (0,0) -- (1, 1);
\draw (1,1) -- (2,0);
\draw (2,0) -- (1, -1);
\draw (1, -1) -- (0,0);
\fill (0,0) circle (2pt) node[left] {$1$};
\fill (1,1) circle (2pt) node[above] {$2$};
\fill (2,0) circle (2pt) node[right] {$3$};
\fill (1,-1) circle (2pt) node[below] {$4$};
\end{tikzpicture}\\
\ \\
\begin{tikzpicture}[scale=.8]
\draw (0,0) -- (1, 1);
\draw (1,1) -- (2,0);
\fill (0,0) circle (2pt) node[left] {$1$};
\fill (1,1) circle (2pt) node[above] {$2$};
\fill (2,0) circle (2pt) node[right] {$3$};
\end{tikzpicture}
\hspace{.5in}
\begin{tikzpicture}[scale=.8]
\draw (0,0) -- (1, 1);
\draw (1,1) -- (2,0);
\fill (0,0) circle (2pt) node[left] {$1$};
\fill (1,1) circle (2pt) node[above] {$3$};
\fill (2,0) circle (2pt) node[right] {$2$};
\end{tikzpicture}
\hspace{.5in}
\begin{tikzpicture}[scale=.8]
\draw (0,1) -- (1, 0);
\draw (1,0) -- (2,1);
\fill (0,1) circle (2pt) node[left] {$1$};
\fill (1,0) circle (2pt) node[below] {$2$};
\fill (2,1) circle (2pt) node[right] {$3$};
\end{tikzpicture}
\hspace{.5in}
\begin{tikzpicture}[scale=.8]
\draw (0,1) -- (1, 0);
\draw (1,0) -- (2,1);
\fill (0,1) circle (2pt) node[left] {$2$};
\fill (1,0) circle (2pt) node[below] {$1$};
\fill (2,1) circle (2pt) node[right] {$3$};
\end{tikzpicture}
\end{center}
\end{ex}

We are now able to state the main result of this paper, which characterizes exactly when a labeling of a transverse non-messing-up poset is sort-invariant in terms of its corner-sets.

\begin{thm}\label{thm:sort-invariance}
Let $P$ be a transverse non-messing-up poset and $\L$ a labeling of $P$.  The labeling $\L$ is sort-invariant if and only if each nonempty corner-set in $P$ is good.
\end{thm}

In other words, sort-invariance of a transverse non-messing-up poset (that is, a convex subposet of a cylinder poset) is characterized by a forbidden subpattern: a bad corner-set.

Theorem~\ref{thm:sort-invariance} is proved in three steps, two of which we highlight as propositions.  The first step is to analyze sort-invariance for the product of two chains (the poset version of the matrix), the second step is to consider convex subposets of the product of two chains, and the final step is to look at the cylinder.

In the next proposition, we describe sort-invariance for posets which are the product of two chains.

\begin{prop}\label{prop:arrays}
Let the poset $P$ be the product of two chains, and let $\L$ be a labeling of $P$.  The labeling $\L$ is sort-invariant if and only if each nonempty corner-set in $P$ is good.
\end{prop}

\begin{proof}
First, note that $P$ is a transverse non-messing-up poset.  Let $\R$ and $\C$ be the rows and columns of $P$, which are analogous to the rows and columns in a matrix.  The row which intersects each column in the minimum position of the column will be called the ``first'' row, the row intersecting in the position above that will be the ``second'' row, and so on.  We similarly name the columns.

We first note that sort-invariance is equivalent to saying that for all $k \in \{1, \ldots, |P|\}$, the shape formed by $\{1, \ldots, k\}$ in $\sr\sc(\L)$ is the same as the shape formed by $\{1, \ldots, k\}$ in $\sc\sr(\L)$.  Thus we can replace the labels in $\L$ by just $\{1_k,2_k\}$, where $1_k$ replaces each element of $\{1, \ldots, k\}$ and $2_k$ replaces each element of $\{k+1, \ldots, |P|\}$, and look at the resulting placements of the $1_k$s after $\sr\sc$ and $\sc\sr$, with the understanding that $2_k > 1_k$.

For the remainder of the proof we will drop the subscript $k$ and just look at a labeling by $1$s and $2$s, and decide when $\sr$ and $\sc$ commute on this filling.

The operation $\sc$ pushes all the $1$s to the bottom of each column.  One interpretation of $\sr\sc$ is that it sorts the set of columns so that the column in which the most $1$s appears is now the first column, the column in which the next most $1$s appears is now the second column, and so on.  The shape of $1$s resulting from $\sc\sr$ can be described analogously.

For these shapes to be equal, we need the $i$th most $1$s appearing in any row to equal the number of columns containing at least $i$ $1$s.  This is equivalent to the following criterion: let $S_{\bm{c}} \subseteq R$ be the set of rows in which column $\bm{c}$ has a $1$; index the columns $\bm{c_1}, \bm{c_2}, \ldots$ so that $\bm{c_{i+1}}$ has at least as many $1$s as $\bm{c_i}$ does; then we must have
\begin{equation}\label{eqn:hierarchy}
S_{\bm{c_1}} \subseteq S_{\bm{c_2}} \subseteq \ldots.
\end{equation}
(Equivalently, the roles of ``row'' and ``column'' can be swapped in the previous discussion.)

For example, the following is a labeling satisfying \eqref{eqn:hierarchy}, where the poset has been rotated to highlight the rows and the columns.
$$\begin{array}{llllll}
2 & 1 & 2 & 2 & 1 & 1\\
2 & 1 & 1 & 2 & 1 & 1\\
2 & 1 & 2 & 2 & 2 & 2\\
2 & 1 & 2 & 2 & 2 & 1\\
2 & 2 & 2 & 2 & 2 & 2\\
2 & 1 & 1 & 2 & 1 & 1\\
\uparrow & \uparrow & \uparrow & \uparrow & \uparrow & \uparrow\\
\bm{c_1} & \bm{c_6} & \bm{c_3} & \bm{c_2} & \bm{c_4} & \bm{c_5} 
\end{array}$$
The $\bm{c_i}$ designations are given below each column.  Note that $\bm{c_1}$ and $\bm{c_2}$ can be swapped.

We claim that the inclusions of \eqref{eqn:hierarchy} are equivalent to having no bad corner-set.  Certainly if the inclusions of \eqref{eqn:hierarchy} hold then there is no such bad corner-set.  Conversely, if there is no bad corner-set then consider any pair of columns $\bm{c}$ and $\bm{c'}$, where $\bm{c}$ has at least as many $1$s as $\bm{c'}$.  If there is a row in which $\bm{c}$ has a $1$ but $\bm{c'}$ does not, then every row in which $\bm{c'}$ has a $1$ must be a row in which $\bm{c}$ does as well.  This implies \eqref{eqn:hierarchy}, and completes the proof.
\end{proof}

\begin{prop}\label{prop:convex}
Let the poset $P$ be a convex subposet of the product of two chains, and let $\L$ be a labeling of $P$.  The labeling $\L$ is sort-invariant if and only if each nonempty corner-set in $P$ is good.
\end{prop}

\begin{proof}
Suppose $P$ is a convex subposet of $P'$, where $P'$ is a product of two chains.  Extend the labeling $\L$ of $P$ to a labeling $\L'$ of $P'$ by giving the label $1$ to every element of $P'\setminus P$ less than an element of $P$, and the label $|P|$ to all other elements of $P' \setminus P$.  Clearly $\L$ is sort-invariant if and only if $\L'$ is sort-invariant.  Similarly, $\L$ has a bad corner-set if and only if $\L'$ has a bad corner-set.  Proposition~\ref{prop:arrays} links these two concepts for $\L'$, and thus we conclude that $\L$ is sort-invariant if and only if all of its non-empty corner-sets are good.
\end{proof}

\begin{proof}[Proof of Theorem~\ref{thm:sort-invariance}]
A labeling of a poset is sort-invariant if and only if the labeling of each connected component of the poset is sort-invariant.  Thus we need only prove the result for a connected poset.  The remainder of the proof resembles that of \cite[Theorem 6]{nmut}.

Let $P$ be a finite convex subposet of $Cyl_{k,n}$.  Cut $Cyl_{k,n}$ to get a preimage of $P$ in the plane.  Glue together copies of this poset via the identifications on the cylinder.  After perhaps removing elements at the sides of the planar poset, this is a convex subposet of a product of two chains.  For the labeling $\L$ of $P$, label every preimage of $x$ in the plane by $\L(x)$.  Glue enough copies of the poset so that after the two sorting procedures, sufficiently many of the centermost copies in the plane have the labels they would have had on the cylinder.  This is possible because only finitely many elements cross over a line of identification.  Call this glued-together poset $P'$, and its labeling $\L'$.

The labeling $\L$ has a bad corner-set if and only if $\L'$ does.  Thus, by Proposition~\ref{prop:convex}, $\L$ has a bad corner-set if and only if $\L'$ is not sort-invariant.  If $\L'$ is sort-invariant then certainly $\L$ is as well.  Therefore, if $\L$ has no bad corner-sets, then $\L$ is sort-invariant.

On the other hand, suppose that $\L$ has a bad corner-set.  Thus $\L'$ has a bad corner-set, and is not sort-invariant.  We must be sure that the sort-invariance fails somewhere in one of the centermost copies of the cut poset.  This will ensure that $\L$ is also not sort-invariant, which will conclude the proof.  Since Proposition~\ref{prop:arrays} characterizes sort-invariance in the plane by forbidding bad corner-sets, interacting with the labels appearing in a bad corner-set during $\sr$ or $\sc$ must lead to the failure of the sort-invariance.  By construction of $P'$, such failure must appear among the centermost copies of the cut poset.  Since these centermost copies have the same labels that they would have had on the cylinder, this means that $\L$ is not sort-invariant.
\end{proof}

\section{Non-split sort-invariance}\label{sec:non-transverse}

One might hope that sort-invariance for non-transverse non-messing-up posets can be characterized by an analogous statement about bad corner-sets.  Unfortunately the obvious generalization of this concept is not the answer.  More precisely, recall the definition of a corner-set, and in particular equation~\eqref{eqn:corner-set}.  In a non-transverse poset, the subsequent equation~\eqref{eqn:corner-set-small} does not necessarily hold, and the corner-set might well have more than four elements.  We can extend the definition of ``bad'' in the obvious way, to say that a corner-set is bad if there exist elements 
\begin{eqnarray*}
x' &\in& \bm{r_x} \cap \bm{c_x},\\
y' &\in& \bm{r_y} \cap \bm{c_y},\\
u &\in& \bm{r_x} \cap \bm{c_y}, \text{ and}\\
v &\in& \bm{r_y} \cap \bm{c_x},
\end{eqnarray*}
so that the labels of these elements either satisfy
$$\L(x'), \L(y') > \L(u), \L(v)$$
or
$$\L(x'), \L(y') < \L(u), \L(v).$$
However, avoidance of such a bad corner-set is not enough to guarantee sort-invariance in a general (non-transverse) non-messing-up poset, as demonstrated in the following example.

\begin{ex}
Figure~\ref{fig:poset-ex1} gives a non-transverse non-messing-up poset $P$ (as described in \cite{nmut}) and a labeling $\L$ of $P$.
\begin{figure}[htbp]
\begin{tikzpicture}[scale=.8]
\draw[densely dashed] (0,0) -- (1,1);
\draw[densely dashed] (.96,1) -- (.96,2);
\draw (1.04,1) -- (1.04,2);
\draw (1,2) -- (0,3);
\draw (0,0) -- (-1,1.5);
\draw[densely dashed] (0,3) -- (-1, 1.5);
\draw[fill] (0,0) circle (2pt) node[below] {$5$};
\draw[fill] (1,1) circle (2pt) node[right] {$1$};
\draw[fill] (1,2) circle (2pt) node[right] {$3$};
\draw[fill] (0,3) circle (2pt) node[above] {$2$};
\draw[fill] (-1,1.5) circle (2pt) node[left] {$4$};
\end{tikzpicture}
\caption{A non-messing-up poset $P$ whose rows are drawn with dashed lines and columns with solid.  A labeling $\L$ of $P$ is also given.}\label{fig:poset-ex1}
\end{figure}
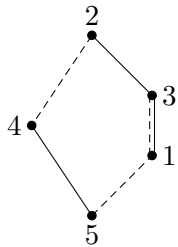

Notice that, in the sense described at the beginning of this section, the labeling $\L$ is good.  However, it is not sort-invariant, as demonstrated in Figure~\ref{fig:poset-ex2}.
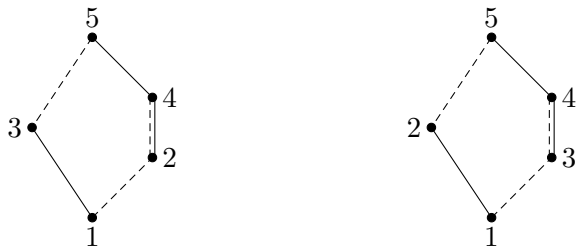
\begin{figure}[htbp]
\begin{tikzpicture}[scale=.8]
\draw[densely dashed] (0,0) -- (1,1);
\draw[densely dashed] (.96,1) -- (.96,2);
\draw (1.04,1) -- (1.04,2);
\draw (1,2) -- (0,3);
\draw (0,0) -- (-1,1.5);
\draw[densely dashed] (0,3) -- (-1, 1.5);
\draw[fill] (0,0) circle (2pt) node[below] {$1$};
\draw[fill] (1,1) circle (2pt) node[right] {$2$};
\draw[fill] (1,2) circle (2pt) node[right] {$4$};
\draw[fill] (0,3) circle (2pt) node[above] {$5$};
\draw[fill] (-1,1.5) circle (2pt) node[left] {$3$};
\end{tikzpicture}
\hspace{1in}
\begin{tikzpicture}[scale=.8]
\draw[densely dashed] (0,0) -- (1,1);
\draw[densely dashed] (.96,1) -- (.96,2);
\draw (1.04,1) -- (1.04,2);
\draw (1,2) -- (0,3);
\draw (0,0) -- (-1,1.5);
\draw[densely dashed] (0,3) -- (-1, 1.5);
\draw[fill] (0,0) circle (2pt) node[below] {$1$};
\draw[fill] (1,1) circle (2pt) node[right] {$3$};
\draw[fill] (1,2) circle (2pt) node[right] {$4$};
\draw[fill] (0,3) circle (2pt) node[above] {$5$};
\draw[fill] (-1,1.5) circle (2pt) node[left] {$2$};
\end{tikzpicture}
\caption{The left-hand poset is $\sr\sc(\L)$, and the right-hand poset is $\sc\sr(\L)$.  They are unequal, and thus $\L$ is not sort-invariant.}\label{fig:poset-ex2}
\end{figure}
\end{ex}

It is disappointing that the equivalence between sort-invariance and avoidance of a bad corner-set does not carry over to this more general setting, and it remains open to characterize sort-invariance of these non-transverse non-messing-up posets.

\section{Preferred sorting procedures for matrices}

We conclude with a brief discussion of a different but related question regarding sorted matrices; namely, whether a given sorted matrix $A$ is more likely to have arisen from $\sr\sc$ or from $\sc\sr$.  Certainly there is a matrix $M$ which can yield $A$ under both of these operations, namely $M=A$, but there may be a significant preference for $\sr\sc$, say, in general.

Say that the matrix is \emph{sorted} if the data in each row are non-decreasing from left to right, and if the data in each column are non-decreasing from top to bottom.  Note that the transpose of a sorted matrix is also necessary sorted. Given a sorted matrix $A$, there are many possible original matrices $M$ with $\sr\sc(M) = A$ or with $\sc\sr(M) = A$.  Furthermore, the set of such $M$ depends on whether the rows were sorted first and then the columns, or vice versa.  For a given sorted matrix, it may be that one of these options is preferred to the other.  

Without loss of generality, suppose that the data in the matrix is the set $\{1, \ldots, rc\}$, where the matrix has $r$ rows and $c$ columns.  Using \cite[Exercise 7.20b]{ec1}, we can compute the number of matrices $M$ for which $\sr\sc(M) = A$, and likewise for which $\sc\sr(M) = A$.  First we define the function $\h_A$ on the set $\{1, \ldots, rc\}$ of entries in $A$: if $i$ is in the first row of $A$ then $\h_A(i) = 1$; otherwise,
$$\h_A(i) = \# \{j < i : j \text{ is in the row immediately above $i$ in $A$ and not to the left of $i$}\}.$$
Set
$$\h_A(A) = \prod \h_A(i).$$
It follows from \cite[Exercise 7.20b]{ec1} that the number of matrices $M$ for which $\sr\sc(M) = A$ is
$$\h_A(A) \cdot (r!)^c \cdot c!.$$
Likewise, the number of $M$ for which $\sc\sr(M) = A$ is
$$\h_{A^T}(A) \cdot (c!)^r \cdot r!,$$
where $A^T$ is the (sorted) transpose of the sorted matrix $A$.

Therefore we can make the following conclusion about preferred sorting procedures.

\begin{prop}\label{prop:preferred}
Suppose that $A$ is a sorted $r$-by-$c$ matrix.  Let $M$ be an $r$-by-$c$ matrix chosen at random, and flip a fair coin to determine whether to sort $M$ by performing $\sr\sc$ or $\sc\sr$.  Then the probability that $A = \sr\sc(M)$ is
$$\frac{\h_A(A) \cdot (r!)^{c-1}}{\h_A(A) \cdot (r!)^{c-1} + \h_{A^T}(A) \cdot (c!)^{r-1}}.$$
\end{prop}

\section*{Acknowledgements}
Sincere gratitude is due to Peter Winkler for very helpful and entertaining discussions.

\end{document}